\documentclass[11pt, a4paper]{article}

\usepackage[utf8]{inputenc}
\usepackage{amsmath, amsthm, amsfonts, amssymb}

\usepackage{mathabx,enumerate}

\usepackage{nicefrac, xfrac}

\usepackage{indentfirst} 

\usepackage{authblk}
\usepackage{lipsum}

\usepackage[backend=biber,maxbibnames=99]{biblatex}
\newtheorem{lemma}{Lemma}[section]
\newtheorem{theorem}[lemma]{Theorem}
\newtheorem{proposition}[lemma]{Proposition}
\newtheorem{corollary}[lemma]{Corollary}

\newtheorem{result}{Theorem}

\newtheorem{corollaryresult}[result]{Corollary}

\theoremstyle{definition}

\newcommand{\abs}[1]{\ensuremath{\left| #1 \right|}}
\newcommand{\op}{\operatorname}
\newcommand{\ce}[2]{\operatorname{C}_{#1}(#2)}
\newcommand{\no}[2]{\operatorname{N}_{#1}(#2)}
\newcommand{\ze}[1]{\operatorname{Z}(#1)}

\newcommand{\rad}[2]{\op{O}_{#1}(#2)}

\newcommand{\rup}[2]{\op{O}^{#1}(#2)}
\newcommand{\syl}[2]{\op{Syl}_{#1}\left(#2\right)}

\newcommand{\groupgen}[1]{\langle #1 \rangle}

\newcommand{\irr}{\operatorname{Irr}}

\newcommand{\gal}{\operatorname{Gal}}

\newcommand{\smid}{\! \mid \!}

\DeclareMathSymbol{\shortminus}{\mathbin}{AMSa}{"39}

\renewcommand{\phi}{\varphi}
\renewcommand{\epsilon}{\varepsilon}

\newcommand{\Q}{\mathbb{Q}}

\newcommand{\ZZ}{{\mathbb{Z}}}

\newcommand{\GL}{\operatorname{GL}}
\newcommand{\PGL}{\operatorname{PGL}}
\newcommand{\SL}{\operatorname{SL}}
\newcommand{\PSL}{\operatorname{PSL}}
\newcommand{\SU}{\operatorname{SU}}
\newcommand{\PSU}{\operatorname{PSU}}
\newcommand{\Sp}{{{\operatorname{Sp}}}}

\newcommand{\SO}{{{\operatorname{SO}}}}
\newcommand{\Spin}{{{\operatorname{Spin}}}}
\newcommand{\PSO}{{{\operatorname{PSO}}}}

\newcommand{\tw}[1]{{}^{#1}\!}

\let\be=\beta

\let\la=\lambda

\title{On the degrees of irreducible characters fixed by some field automorphism in finite groups}

\date{\today}

\author{Nicola Grittini \thanks{The paper was written while the author was visiting the Department of Mathematics in the RPTU, in Kaiserslautern. The author is grateful for the generous hospitality and the pleasant environment he found in the University.}}

\begin{document}

\maketitle

\begin{abstract}
We prove a variant of the Theorem of Ito-Michler, investigating the properties of finite groups where a prime number $p$ does not divide the degree of any irreducible character left invariant by some Galois automorphism $\sigma$ of order $p$.
\end{abstract}

\section{Introduction}

This paper is meant to be the continuation of the work the author began in \cite{Grittini:Degrees_of_irreducible_fixed_p-solvable}. It was proved in \cite{Grittini:Degrees_of_irreducible_fixed_p-solvable} that, if $G$ is a $p$-solvable group, for some prime number $p$, $\sigma \in \gal(\mathbb{Q}_{\abs{G}}/\mathbb{Q})$ is of order $p$, and $p$ does not divide the degree of any irreducible character of $G$ fixed by $\sigma$, then $G$ has a normal Sylow $p$-subgroup. In $p$-solvable groups, this extends \cite[Theorem~A]{Dolfi-Navarro-Tiep:Primes_dividing}, which is itself a variant of the celebrated Theorem of Ito-Michler.

As it was already noted in \cite{Grittini:Degrees_of_irreducible_fixed_p-solvable}, the hypothesis on group $p$-solvability is essential, since it is possible to find simple groups, of order divisible by a prime number $p$, where $p$ does not divide the degree of any $\sigma$-invariant irreducible character of the group, for some Galois automorphism $\sigma$ of order $p$. Examples of this kind are $\op{PSL}_2(8)$ and $\op{PSL}_2(32)$ with $p=3$, or $\op{PSL}_2(81)$ and $\op{Sz}(32)$ for $p=5$.

Nevertheless, in this paper we prove that it is still possible to extend the results of \cite{Grittini:Degrees_of_irreducible_fixed_p-solvable} to non-solvable groups, and our hypothesis on $\sigma$-invariant character degrees imposes severe restrictions to the group normal structure also in the non-solvable case.

\begin{result}
\label{result:main}
Let $G$ be a finite group, let $p$ be a prime number and let $\sigma \in \gal(\mathbb{Q}_{\abs{G}}/\mathbb{Q})$ of order $p$. Suppose that $p$ does not divide the degree of any irreducible character of $G$ fixed by $\sigma$, then $\rup{p'}{G} = \rad{p}{G} \times K$ for some $K \lhd G$ such that $\rad{p'}{K}$ is solvable and $K/\rad{p'}{K}$ is the direct product of some non-abelian simple groups of Lie type with no $\sigma$-invariant characters of degree divisible by $p$. Moreover, if $p=2$ then $G$ has a normal Sylow $2$-subgroup.
\end{result}

We may notice that in general we cannot hope for $\rad{p'}{K}$ to be trivial, since $\op{SL}_2(8)$ would be a counterexample for $p=3$. Hence, this is probably the closest we can be to have a normal Sylow $p$-subgroup in a non-$p$-solvable group.

We can see that, if $G$ is a finite group and $p$ is an odd prime number, it is possible to take $\sigma \in \gal(\mathbb{Q}_{\abs{G}}/\mathbb{Q})$ of order $p$ of minimal kernel, such that, for each prime number $r$ dividing $\abs{G}$, the restriction of $\sigma$ to $\Q_{\abs{G}_r}$ is equal the identity if and only if $p \nmid [\Q_{\abs{G}_r} : \Q]$. Then, we have that a character $\chi \in \irr(G)$ is $\sigma$-invariant if and only if $p \nmid [\Q(\chi) : \Q]$ and, therefore, either $G$ verifies the hypothesis of our Theorem~\ref{result:main}, or there exists $\chi \in \irr(G)$ such that $p$ divides $\chi(1)$ and not $[\Q(\chi) : \Q]$. Hence, the following corollary holds.

\begin{corollaryresult}
\label{corollaryresult:field_index}
Let $G$ be a finite group and let $p$ be an odd prime number. Suppose that, for each $\chi \in \irr(G)$, if $p$ divides $\chi(1)$ then it also divides $[\Q(\chi) : \Q]$. Then, $\rup{p'}{G} = \rad{p}{G} \times K$ for some $K \lhd G$ such that $\rad{p'}{K}$ is solvable and $K/\rad{p'}{K}$ is the direct product of some non-abelian simple groups of Lie type.
\end{corollaryresult}

This may be compared with the results and conjectures in \cite{Navarro-Tiep:The_fields_of_values}. In fact, if $G$ is a finite group and $p$ a prime number, in \cite{Navarro-Tiep:The_fields_of_values} is conjectured that, for $\chi \in \irr(G)$, if $p \nmid \chi(1)$, then, roughly speaking, the $p$-part of $[\Q(\chi) : \Q]$ is relatively \textit{large}. It is also proved that, if $p \mid \chi(1)$, then $[\Q(\chi) : \Q]$ can be any value (see \cite[Theorem~2.2]{Navarro-Tiep:The_fields_of_values}). What we prove here is that, if $G$ is a finite group and $p$ an odd prime number, then either the thesis of our Theorem~\ref{result:main} holds, or there exists at least one character $\chi \in \irr(G)$ such that $p \mid \chi(1)$ and the $p$-part of $[\Q(\chi) : \Q]$ is trivial. Of course, this does not contradict \cite[Theorem~2.2]{Navarro-Tiep:The_fields_of_values}, since our theorem guarantees the existence of merely \emph{one} irreducible character with said characteristics in the group.

Essential for proving our Theorem~\ref{result:main} is to categorize, at east partially, the simple groups which have a $\sigma$-invariant character of degree divisible by $p$, for every prime number $p$ dividing the order of the group and for every Galois automorphism $\sigma$ of order $p$. The following theorem partially classifies the finite non-abelian simple groups with a rational character of degree divisible by a prime number $p$. The theorem is due to Gunter Malle, and it is published here with his consent.

For coprime integers $p,q$, we let $e_p(q)$ denote the multiplicative order of $q$ modulo~$p$.

\begin{result}
\label{result:simple_rational}
 Let $S$ be a finite simple group of Lie type and $p>2$ a prime divisor
 of~$\abs{S}$. If all rational characters of $S$ have $p'$-degree, then
 one of the following statements is true:
\begin{enumerate}[\rm(1)]
\item $S=\tw2B_2(q^2)$ with $p \mid (q^4+1)$,
\item $S=\PSL_n(q)$ with $e_p(q)=1$ and $p>n$,
\item $S=\PSU_n(q)$ with $e_p(q)=2$ and $p>n$,
\item $S=\PSL_n(q)$ with $n=2$, or $(n,e_p(q)) = (4,2)$ or
$(n,p,e_p(q))=(3,3,1)$,
\item $S=\PSU_n(q)$ with $(n,e_p(q)) = (4,1)$ or
$(n,p,e_p(q))=(3,3,2)$,
\item $S=\PSO_{2n}^+(q)$ with $(n,e_p(q))\in\{(5,2),(7,2)\}$, or
\item $S=\PSO_{2n}^-(q)$ with
$$(n,e_p(q))\in\{(4,1),(4,2),(4,4),(5,1),(6,1),(6,2),(7,1),(8,1),(8,2)\}.$$
\end{enumerate}
\end{result}

Of course, the thesis of Theorem~\ref{result:simple_rational} is stronger that what we actually need, and it is possible to find a simple group $S$, of order divisible by $p$, which does not have any rational character of degree divisible by $p$ and, nevertheless, it has a $\sigma$-invariant character of degree divisible by $p$ for every choice of $\sigma \in \gal(\mathbb{Q}_{\abs{S}}/\mathbb{Q})$ of order $p$. In fact, $S=\PSL_3(4)$ is an example of this kind of groups, for $p=3$, since it has no rational irreducible characters of degree divisible by $3$ and, nevertheless, it has only $2$ irreducible characters of degree equal to $45$ and only $2$ of degree equal to $63$, which are therefore fixed by any Galois automorphism of order $3$. In general, if $S = \PSL_3(q)$ such that $3 \mid q-1$, and $p=3$, we prove in Proposition~\ref{proposition:PSL3} that $S$ always has a $\sigma$-invariant character of degree divisible by $p$, despite appearing as an exception in Theorem~\ref{result:simple_rational}.

We should mention here that it is particularly difficult to find examples of simple groups of Lie type of rank higher then $1$ and such that $p$ does not divide the degree of any $\sigma$-invariant character. In fact, for example, if $G=\op{SL}_n(q)$ and $r$ is a primitive prime divisor of $q^n - 1$, then $G$ has a semisimple element $s$ of order $r$ and such that $q-1$ divides $\abs{G:\ce{G}{s}}$ (see for instance \cite[Lemma~2.4]{Moreto-Tiep:Prime_divisors}), and the Deligne-Lusztig theory suggests the existence of an irreducible character in the dual group of degree $\abs{G:\ce{G}{s}}$ (or possibly a fraction of $\abs{G:\ce{G}{s}}$) having values in $\Q_r$. Hence, if a prime number $p$ divides $q-1$ and not $r-1$, the group has an irreducible character of degree divisible by $p$ and fixed by every Galois automorphism $\sigma$ of order $p$. Now, it is possible to find a group $G=\op{PSL}_n(q)$ and a prime number $p$ such that $p \mid q-1$ and $p \mid r-1$ for every primitive prime divisor $r$ of $q-1$, however, the smallest example we could find of rank higher then $1$ is $\op{PSL}_4(11)$, with $p=5$, which is itself quite large and, nevertheless, we verified that it has rational characters of degree divisible by $p$. This shows that it is computationally difficult to find high-rank simple groups which verify the hypotheses of our theorem.

Finally, as a by-product of some of the results we use in order to prove our Theorem~\ref{result:main}, we are able to prove the following theorem on groups where all rational irreducible characters have odd degree. Notice that groups with these characteristics were already studied in \cite{Tiep-TongViet:Odd-degree_rational}.

\begin{result}
\label{result:rational}
Let $G$ be a finite non-solvable group and suppose that all the rational characters of $G$ have odd degree. Suppose also that there is no subgroup $N \lhd \rup{2'}{G}$ such that $\rup{2'}{G}/N \cong D_{10}$, the dihedral group of $10$ elements. Then, there exists $L \lhd G$ solvable such that $L < \rup{2'}{G}$ and $\rup{2'}{G}/L \cong S_1 \times ... \times S_k$, with $S_i \cong \op{PSL}_2(3^{f_i})$, for some odd $f_i \geq 3$, for each $i=1,...,k$.
\end{result}

This result may be compared with \cite[Theorem~C]{Dolfi-Malle-Navarro:Grous_no_real_p-elements}, also reported as \cite[Theorem~3.2]{Tiep-TongViet:Odd-degree_rational}, and it underlines the relation between rational characters of odd degree and rational (or real) elements of odd order. The topic is also studied by the author in \cite{Grittini:Odd-degree_rational}.

\section{Simple groups and rational characters}
\label{section:simple_groups}

The results in this section are due to Gunter Malle, and they are published here with his consent. The two propositions in this section, taken together, prove our Theorem~\ref{result:simple_rational}.

For coprime integers $p,q$ we let $e_p(q)$ denote the multiplicative order
of $q$ modulo~$p$.
 
For the following result note that unipotent characters of finite reductive
groups $G$ have the centre $\ze{G}$ in their kernel and so descend to irreducible
characters of $G/\ze{G}$, which by a slight abuse of notation we will also call
unipotent.

\begin{proposition}
\label{proposition:unip_rational}
 Let $S$ be a finite simple group of Lie type and $p>2$ a prime divisor
 of~$\abs{S}$. If all rational unipotent characters of $S$ have $p'$-degree, then
 one of the following statements is true:
\begin{enumerate}[\rm(1)]
\item $S=\tw2B_2(q^2)$ with $p \mid (q^4+1)$,
\item $S=\PSL_n(q)$ with $e_p(q)=1$ and $p>n$,
\item $S=\PSU_n(q)$ with $e_p(q)=2$ and $p>n$,
\item $S=\PSL_n(q)$ with $n=2$, or $(n,e_p(q))\in\{(3,3),(4,2)\}$ or
$(n,p,e_p(q))=(3,3,1)$,
\item $S=\PSU_n(q)$ with $(n,e_p(q))\in\{(3,6),(4,1)\}$ or
$(n,p,e_p(q))=(3,3,2)$,
\item $S=\PSO_{2n}^+(q)$ with $(n,e_p(q))\in\{(5,2),(7,2)\}$, or
\item $S=\PSO_{2n}^-(q)$ with
$$(n,e_p(q))\in\{(4,1),(4,2),(4,4),(5,1),(6,1),(6,2),(7,1),(8,1),(8,2)\}.$$
\end{enumerate}
\end{proposition}

\begin{proof}
For the Tits simple group this is immediate from the known character table,
so we may now assume that $S=G/\ze{G}$ for $G$ the group of fixed points under a
Steinberg map of a simple algebraic group of simply connected type. By our
prior remark we may argue with the unipotent characters of $G$ in place of $S$.
If $p$ is the defining prime of~$G$, then the Steinberg character of $G$, of
degree a power of $p$ \cite[Proposition~3.4.10]{Geck-Malle:Character_Theory}, is as required. So now assume
$p$ is not the defining prime. If $G$ is of exceptional type, it is immediate
from the known degrees (listed e.g.~in \cite[Section~13.9]{Carter:Finite_groups_of_Lie_type}) that there do exist
rational unipotent characters of degree divisible by~$p$ for any non-defining
prime $p$, unless when $G$ is a Suzuki group $\tw2B_2(q^2)$ and $p$ divides
$q^4+1$.   \par
For the groups of classical type we use that all unipotent characters are
rational, see \cite[Corollary~4.5.6]{Geck-Malle:Character_Theory}. Also recall from
\cite[Proposition~4.5.9]{Geck-Malle:Character_Theory} that degrees of unipotent characters are products of
cyclotomic polynomials $\Phi_d$ (evaluated at $q$), possibly divided by a
2-power (which does not matter for our purpose here since we assume $p>2$).
Furthermore, as a consequence of $d$-Harish-Chandra theory, any unipotent
character lying in the $d$-Harish-Chandra series above a $d$-cuspidal
unipotent character $\rho$ has at least the same $\Phi_d$-part in its generic
degree as $\rho$ \cite[Corollary~4.6.25]{Geck-Malle:Character_Theory}.

Now first assume $G=\SL_n(q)$ with $n\ge3$. Let $d=e_p(q)$ be the order of $q$
modulo~$p$. Note that $d\le n$ if $p$ divides $\abs{G}$. If $\lambda$ is a $d$-core
partition of $n$, then by the degree formula \cite[Proposition~4.3.1]{Geck-Malle:Character_Theory} the
degree of the unipotent character labelled by $\lambda$ is divisible by
$\Phi_d(q)$ and hence by $p$. According to the main result of \cite{Granville-Ono:Defect_zero_p-blocks} there
is such a $d$-core for any integer $n\ge d$ when $d\ge4$. So now assume $d\le3$.
There exist 3-cores $\la_n\vdash n$ for $n=4,5,6$. Now if $n\equiv t\pmod3$ with
$t\in\{4,5,6\}$ then all unipotent characters in the $3$-Harish-Chandra series
of $G$ above the 3-core $\la_t$ have degree divisible by $\Phi_3(q)$, hence
by~$p$. This only leaves the case $n=3$ when $d=3$. There are 2-core partitions
$\la_n\vdash n$ for $n=3,6$, and since any unipotent character in a
2-Harish-Chandra series above the characters labelled by these cores has degree
divisible by~$p$, we are left with $n=4$ when $d=2$. Finally, in the case $d=1$
assume that $p\le n$. Then $p$ divides $\Phi_p(q)$ and by our previous
considerations there is a unipotent character of degree divisible by~$p$, unless
possibly when $n=3$.
\par
Next consider $G=\SU_n(q)$ with $n\ge3$ and let $d=e_p(-q)$. Since the degree
formula for unipotent characters of $G$ is obtained from the one for $\SL_n(q)$
by replacing $q$ by $-q$ (and adjusting the sign if necessary)
\cite[Proposition~4.3.5]{Geck-Malle:Character_Theory}, our investigation for $\SL_n(q)$ only leaves us with
the cases $(n,d)\in\{(3,6),(4,1)\}$ and $d=2$, $p>n$, as listed in~(3).
\par
Now let $G_n=\Sp_{2n}(q)$ or $\SO_{2n+1}(q)$ with $n\ge2$ and set $d=e_p(q)$.
If $d$ is odd, the for any $d$-core $\la\vdash n$ the unipotent character
labelled by the symbol of defect~1 for the bipartition $(\la;-)$ is of
$p$-defect zero \cite[Corollary~4.4.18]{Geck-Malle:Character_Theory}. So as before we only need to concern
ourselves with the cases $d=1$ and $d=3$. Now the cuspidal unipotent character
of $\Sp_4(q)$ has degree divisible by $\Phi_1(q)$, and so has every unipotent
character of $G_n$, $n\ge3$, in its Harish-Chandra series. Moreover, the degrees
of the unipotent characters of $G_n$, $n=3,4,5$, labelled by the bipartitions
$(2;1)$, $(31;-)$, respectively $(31;1)$, are divisible by $\Phi_3(q)$. Now
assume
$d=2e$ is even. If $e$ is odd we use the fact (Ennola duality) that the set of
degree polynomials of $G_n$ is invariant under replacing $q$ by $-q$ (and
adjusting the sign), so if there exists a unipotent character of $G_n$ of degree
divisible by $\Phi_e(q)$, then there is also one of degree divisible by
$\Phi_{2e}(q)$. If $e$ is even, then let $\la=(\la_1\le\ldots\le\la_r)\vdash n$
be an $e$-core, and $\be=(\be_1,\ldots,\be_r)$ with $\be_i:=\la_i+i-1$ be
the corresponding $\be$-set. Let $\gamma=(0,\ldots,r-2)$, a $\be$-set for the
empty partition. Let $S$ be the symbol whose first row consists of the
$\be_i$ with $\be_i\pmod{2e}\in\{0,\ldots,e-1\}$ as well as the $\gamma_i$ with
$\gamma_i\pmod{2e}\in\{e,\ldots,2e-1\}$ (where we consider the smallest
non-negative residue), and the second row consisting of the remaining $\be_i$,
$\gamma_i$. Then $S$ has rank~$n$ and odd defect and by construction
is an $e$-cocore, hence it labels a unipotent character of $G_n$ of $p$-defect
zero \cite[Corollary~4.4.18]{Geck-Malle:Character_Theory}. This leaves the case $e=2$ for which there might
be no $e$-core of $n$. But here the unipotent characters of $G_n$, $n=2,3$,
labelled by the symbols $\binom{0\ 1}{2}$, $\binom{0\ 1}{3}$ respectively have
degree divisible by~$\Phi_4(q)$, so by~$p$, and we are again done by
$d$-Harish-Chandra theory.
\par
Finally we consider the even-dimensional orthogonal groups
$G_n^\pm=\Spin_{2n}^\pm(q)$ for $n\ge4$. Let $d=e_p(q)$.
If $d$ is odd, the principal series unipotent character labelled by a $d$-core
$\la\vdash n$ is as desired, again by \cite[Corollary~4.4.18]{Geck-Malle:Character_Theory}. For $d=1$ note
that $G_4^+$ has a cuspidal
unipotent character, hence of degree divisible by $\Phi_1(q)$, and so has
$G_9^-$. This leaves the cases in~(5). Let $d=2e$ be even. If $e$ is odd, the
claim follows again by Ennola duality from the case when $d$ is odd. If $e$ is
even, we construct a symbol of defect~0 or~4 which is an $e$-cocore from an
$e$-core of $n$ by the same procedure as for types $B_n$ and $C_n$. Again,
this leaves the case $e=2$. Here, the symbols
$$\binom{3}{1},\ \binom{1\ 5}{0\ 1},\ \binom{0\ 1\ 5}{1},\ \binom{1\ 5}{}$$
are 2-cocores for $G_4^+,G_5^+,G_5^-,G_6^-$ respectively, leaving the case
$(n,d)=(4,4)$ for $G_4^-$, listed in~(5).
\end{proof}

\begin{proposition}
\label{proposition:some_cases}
 Assume we are in one of the following cases of Proposition~{\rm\ref{proposition:unip_rational}}:
 \begin{enumerate}[\rm(1)]
  \item $S=\PSL_3(q)$ with $e_p(q)=3$, or
  \item $S=\PSU_3(q)$ with $e_p(q)=6$.
 \end{enumerate}
 Then $S$ has a rational irreducible characters of degree divisible by~$p$.
\end{proposition}

\begin{proof}
For $S=\PSL_3(q)$ with $q$ odd there is an involution $s$ in the group
$\PGL_3(q)$ with centraliser $\GL_2(q)$. The semisimple characters in the
Lusztig series of $\SL_3(q)$ labelled by $s$ have degree
$\abs{\PGL_3(q):\GL_2(q)}_{q'}=q^2+q+1$, they are rational as any involution is
rational, and they have $\ze{\SL_3(q)}$ in their kernel as $s\in\PSL_3(q)$.
If $q$ is even, let $s$ be the image in $\PGL_3(q)$ of an element of order three
in $\GL_3(q)$ with three distinct eigenvalues. Then $s$ is rational and has
centraliser a maximal torus, of order $q^2-1$ when $3 \mid (q+1)$, respectively
$(q-1)^2.3$ when $3 \mid (q-1)$. The corresponding semisimple character of degree
$q^3-1$, respectively $\Phi_2\Phi_3/3$, of $\SL_3(q)$ is then rational and
trivial on the centre. This deals with the cases $d=3$. The case of
$\PSU_3(q)$ with $d=6$ can be handled entirely analogously.
\end{proof}

\section{Almost-simple groups}

In this section, we will prove Theorem~\ref{result:main} for almost-simple groups, i.e., for groups $G$ such that $S \leq G \leq \op{Aut}(S)$, where the non-abelian simple group $S$ is identified with $\op{Inn}(S)$. In particular, we will prove that, if $G$ is an almost-simple group with socle $S$, $p$ is a prime number such that $p \mid \abs{G/S}$, and $\sigma \in \gal(\Q_{\abs{G}}/\Q)$ has order $p$, then $G$ has a $\sigma$-invariant irreducible character of degree divisible by $p$ and such that $S$ is not contained in its kernel. This last condition is actually more then what we need, and it requires some further work to be obtained, since we cannot simply apply Theorem~A and Theorem~C of \cite{Grittini:Degrees_of_irreducible_fixed_p-solvable}. On the other hand, since there are actually some imprecisions in the proof of Theorem~A and Theorem~C of \cite{Grittini:Degrees_of_irreducible_fixed_p-solvable}, we can see this as an opportunity to correct them.


First, we need to mention a result from \cite{Grittini:Degrees_of_irreducible_fixed_p-solvable}, which we will use repeatedly in the current paper.

\begin{lemma}[{\cite[Lemma~2.4]{Grittini:Degrees_of_irreducible_fixed_p-solvable}}]
\label{lemma:normal_coprime}
Let $G$ be a finite group, let $N \lhd G$ and let $\phi \in \irr(N)$ such that $(\abs{I_G(\phi)/N},o(\phi)\phi(1))=1$. Let $\sigma \in \gal(\mathbb{Q}_{\abs{G}} / \mathbb{Q})$ and suppose there exists $g \in G$ such that $\phi^{\sigma} = \phi^g$. Then, there exists $\chi \in \irr(G \smid \phi)$ such that $\chi^{\sigma} = \chi$. Moreover, $\chi$ can be chosen such that $\chi(1) = \phi(1) \abs{G:I_G(\phi)}$.
\end{lemma}

We will now prove the following, rather technical Proposition~\ref{proposition:semisimple_characters}, which is essentially a slight variant of \cite[Theorem~3.3]{Grittini:Degrees_of_irreducible_fixed_p-solvable}. In doing so, we correct the aforementioned imprecisions appearing in \cite{Grittini:Degrees_of_irreducible_fixed_p-solvable}. Clearly, this would require us to repeat large sections of the proof of \cite[Theorem~3.3]{Grittini:Degrees_of_irreducible_fixed_p-solvable}. We will also rely often on the results of the Deligne-Lusztig theory. A complete description of the theory can be found in \cite{Geck-Malle:Character_Theory}.

First, however, we need to prove the following lemma.


\begin{lemma}
\label{lemma:torus_subgroup}
Let $S = \mathbf{S}^F$ be a finite group of Lie type defined over a field of $q_0$ elements, for some Steinberg map $F$, such that $S$ is perfect and $S/\ze{S}$ is simple. Take $q$ such that $q^2 = q_0$ if $S/\ze{S}$ is of twisted type, and take $q = q_0$ if $S/\ze{S}$ is untwisted. Then, there exists an $F$-stable maximal torus $\mathbf{T} \leq \mathbf{S}$ such that $T = \mathbf{T}^F$ contains $\ze{S}$ and it also has a nontrivial cyclic subgroup $T_0$ of order either $q - 1$ or $q^2 - 1$, and such that $\abs{T_0 \cap \ze{S}}$ is a power of $2$.
\end{lemma}

\begin{proof}
It follows from the Classification of the finite simple groups that, for any choice of $S$, there exists a maximally split torus $T$ containing $\ze{S}$ and containing also a subgroup $T_0$ of order either $q - 1$ or $q^2 - 1$ (see for instance the classification in \cite{Wilson:Finite_Simple_Groups}). If $\abs{\ze{S}}$ is a power of $2$, then we are done. If $S/\ze{S} \cong \op{PSL}_n(q)$ with $n > 2$ and $q \neq 2$, then we can consider the subgroup of $T$ generated by $\op{diag}(\lambda,1,...,1,\lambda^{\shortminus 1})$, for some $\lambda \in \mathbb{F}_q^{\times}$ of order $q - 1$, which has trivial intersection with $\ze{S}$. If $S/\ze{S} \cong \op{PSL}_n(2)$, then $\abs{\ze{S}} = 1$. If $S/\ze{S} \cong \op{PSL}_2(9)$, then $q - 1 = 8$.

If $S/\ze{S} \cong \op{PSU}_n(q)$, then usually $T_0$ can be chosen of order $q - 1$ and, since $\abs{\ze{S}} \mid q + 1$, it follows that $\abs{T_0 \cap \ze{S}} \leq 2$. In the remaining cases, namely if $q = 2$ or $S/\ze{S} \cong \op{PSU}_4(3)$, then $n \geq 4$ and we can consider the subgroup of $T$ generated by $\op{diag}(\lambda,\lambda^{\shortminus 1},1,...,1,\overline{\lambda},\overline{\lambda^{\shortminus 1}})$, for some $\lambda \in \mathbb{F}_{q^2}^{\times}$ of order $q^2 - 1$, which has trivial intersection with $\ze{S}$ (notice that $\op{SU}_4(2)$ has trivial centre).

If $S/\ze{S} \cong \Omega_7(3)$ or $G_2(3)$, then $T$ has a subgroup of order $q - 1 = 2$. If $S/\ze{S} \cong E_6(q)$, then we can see from the description of $T$ given in \cite[Section~4.10.3]{Wilson:Finite_Simple_Groups} that it has a subgroup of order $q - 1$ with trivial intersection with the centre. Finally, if $S/\ze{S} \cong {}^2E_6(q)$, then we can see from \cite[Proposition~2.1]{Buturlakin:Spectra_finite_simple_groups_E6_2E6} that there exists a torus $T \cong (\ZZ_{q+1})^2 \times (\ZZ_{q^2-1})^2$, containing $\ze{S}$, and it has a cyclic subgroup $T_0$ of order $q_0 - 1$, while $\abs{\ze{S}} = (3,q_0 + 1)$. Hence, $T_0 \cap \ze{S} = 1$.
\end{proof}

We recall that, if $S = \mathbf{S}^F$ is a finite group of Lie type, it is always possible to find a finite group $S \leq \tilde{S}$ such that $\tilde{S} = \mathbf{\tilde{S}}^{\tilde{F}}$, for some $\mathbf{S} \leq \mathbf{\tilde{S}}$ with connected centre and some Steinberg map $\tilde{F}$ such that $\tilde{F} = F$ on $\mathbf{S}$ (see \cite[Lemma~1.7.3]{Geck-Malle:Character_Theory}). With a little abuse of terminology, we may say that $\tilde{S}$ is a \emph{regular embedding} for $S$. Moreover, if $S \neq {}^2F_4(2)'$ is a simple group of Lie type, then it is always possible to find a finite group of Lie type $H = \mathbf{H}^F$ such that $S \cong H/\ze{H}$. If $\tilde{H}$ is a regular embedding for $H$, then we can abuse again the terminology and say that $\tilde{S}=\tilde{H}/\ze{H}$ is again a \emph{regular embedding} for $S$.

We may also notice that, if $\nu$ is a field automorphism for a group of Lie type $S$, and $\tilde{S}$ is a \emph{regular embedding} for $S$, then $\nu$ acts also on $\tilde{S}$.

\begin{proposition}
\label{proposition:semisimple_characters}
Let $S = \mathbf{S}^F$ be a finite group of Lie type, for some Steinberg map $F$, such that $S$ is perfect and $S/\ze{S} \ncong {}^2E_6(q)$ is simple, and let $S \leq \tilde{S}$ be a regular embedding for $S$. Let $\nu \in \op{Aut}(S)$ be a field automorphism for $S$ of order $p$, for some odd prime number $p$, and let $\sigma \in \gal(\mathbb{Q}_{p\abs{\tilde{S}}} / \mathbb{Q})$ of order $p$. Then, there exists a semisimple character $\chi \in \irr(\tilde{S})$ such that $\ze{S} \leq \ker(\chi)$, $[\chi_S,{\chi_S}^{\nu}]=0$, and either $\chi^{\sigma}=\chi$ or $\chi^{\sigma}=\chi^{\nu^k}$ for some positive integer $k$ coprime with $p$.
\end{proposition}

\begin{proof}
Suppose that $S$ is defined over a field of $q_0=\ell^f$ elements, then by Lemma~\ref{lemma:torus_subgroup} here exists an $F$-stable maximal torus $\mathbf{T} \leq \mathbf{S}$, such that $T = \mathbf{T}^F$ contains $\ze{S}$, and a subgroup $T_0 \leq T$ of order $q - 1$, with either $q_0 = q$ or $q_0 = q^2$, such that $\abs{T_0 \cap \ze{S}}$ is a power of $2$. By \cite[Lemma~1.7.7]{Geck-Malle:Character_Theory}, there exists an $\tilde{F}$-stable maximal torus $\mathbf{T} \leq \mathbf{\tilde{T}}$ of $\mathbf{\tilde{S}}$ such that, if $\tilde{T} = \mathbf{\tilde{T}}^{\tilde{F}}$, then $\tilde{S} = \tilde{T}S$ and $T = \tilde{T} \cap S$.

Notice that $\nu(x) = x^{\ell^{\nicefrac{f}{p}}}$ for all $x \in T$. By \cite[Lemma~3.1]{Grittini:Degrees_of_irreducible_fixed_p-solvable} and its proof, there exists $\theta \in \irr(\tilde{T})$ such that $o(\theta_{T_0})$ is odd, ${\theta_T}^{\nu} \neq \theta_T$ and $\theta$ is fixed either by $\sigma$ or by $\nu^k \times \sigma^{\shortminus 1}$, for some positive integer $k$ coprime with $p$. Moreover, since $T_0 \cap \ze{S} \leq \ker(\theta)$, we can see from the proof of \cite[Lemma~3.1]{Grittini:Degrees_of_irreducible_fixed_p-solvable} that we can actually take $\theta$ such that it contains the whole centre $\ze{S}$ in its kernel.

Let $\rho \in P \times \groupgen{\sigma}$ such that $\theta^{\rho} = \theta$ and either $\rho = 1 \times \sigma$ or $\rho = \nu^k \times \sigma^{\shortminus 1}$, and let $\chi \in \irr(\tilde{S})$ be a semisimple character lying over $\theta$. By \cite[Lemma~3.2]{Grittini:Degrees_of_irreducible_fixed_p-solvable}, $\chi^{\rho}$ is a semisimple character lying over $\theta^{\rho}=\theta$ and, since $\mathbf{\tilde{S}}$ has connected centre, by \cite[paragraph 2.6.10]{Geck-Malle:Character_Theory} we have that $\chi^{\rho} = \chi$. Finally, notice that $\theta_T, {\theta_T}^{\nu} \in \irr(T)$ are not conjugated in $S$, since otherwise the Frattini argument would imply that $\nu \in I_G(\theta_T) S$ and, thus, $\nu \in I_G(\theta_T)$. Since the irreducible constituents of $\chi_S$ and ${\chi^{\nu}}_S$ are semisimple characters lying over $\theta_T$ and ${\theta_T}^{\nu}$, it follows from \cite[Theorem~2.6.2]{Geck-Malle:Character_Theory} that $[{\chi^{\nu}}_S,\chi_S] = 0$.
\end{proof}

\begin{corollary}
\label{corollary:semisimple_characters}
Let $S \ncong {}^2E_6(q),{}^2F_4(2)'$ be a finite simple group of Lie type, and let $S \leq \tilde{S}$ be a regular embedding for $S$. Let $\nu \in \op{Aut}(S)$ be a field automorphism for $S$ of order $p$, for some odd prime number $p$, and let $\sigma \in \gal(\mathbb{Q}_{p\abs{\tilde{S}}} / \mathbb{Q})$ of order $p$. Then, there exists a semisimple character $\chi \in \irr(\tilde{S})$ such that $[\chi_S,{\chi_S}^{\nu}]=0$ and either $\chi^{\sigma}=\chi$ or $\chi^{\sigma}=\chi^{\nu^k}$ for some positive integer $k$ coprime with $p$.
\end{corollary}

\begin{proof}
Since $S \ncong {}^2F_4(2)'$, there exists a finite group of Lie type $H = \mathbf{H}^F$, for some Steinberg map $F$, such that $H$ is perfect and $S \cong H/\ze{H}$. Then, Proposition~\ref{proposition:semisimple_characters} guarantees the existence of a character $\chi$ with the desired properties which contains $\ze{H}$ in its kernel.
\end{proof}

We now need to address explicitly two of the cases left out by Theorem~\ref{result:simple_rational}.

\begin{proposition}
\label{proposition:PSL3}
Let $S$ be a simple group and assume that either
\begin{enumerate}[\rm(1)]
\item $S=\PSL_3(q)$ with $3 \mid q - 1$, or
\item $S=\PSU_3(q)$ with $3 \mid q + 1$.
\end{enumerate}
Let $G \geq S$ be an almost-simple group with socle $S$ such that $G/S$ is a $3$-group, and let $\sigma \in \gal(\mathbb{Q}_{\abs{G}}/\mathbb{Q})$ be of order $3$. Then, $G$ has a $\sigma$-invariant irreducible character of degree divisible by $3$ which does not contain $S$ in its kernel.
\end{proposition}

\begin{proof}
We prove the proposition explicitly only for $S=\PSL_3(q)$ and $3 \mid q - 1$, since the case of $S=\PSU_3(q)$ with $3 \mid q + 1$ can be handled analogously.

We can identify $G$ with a subgroup of $\op{Aut}(S)$ containing $S$, and we can define $S < \tilde{S} < \op{Aut}(S)$ such that $\tilde{S} \cong \op{PGL}_3(q)$. If we consider the character tables of $\op{PGL}_3(q)$ (see \cite[Table~VIII]{Steinberg:The_representations_of}) and of $\PSL_3(q)$ (see \cite[Table~2]{Simpson-Frame:The_character_table_for}), with $3 \mid q-1$, we can see that almost all the irreducible characters of $\tilde{S}$ restrict irreducibly to $S$, with the only exception being a single character of $\tilde{S}$ of degree $(q+1)(q^2+q+1)$. We can also notice that $\tilde{S}$ has exactly $\frac{1}{6}(q^2 - 5q + 10)$ irreducible characters of degree $(q+1)(q^2+q+1)$, while $S$ has exactly $\frac{1}{3}(q - 4)$ irreducible characters of degree $q^2+q+1$ and $\frac{1}{6}q(q - 1)$ irreducible characters of degree $(q-1)(q^2+q+1)$. We can see that, since $q \equiv 1$ (mod $3$), then $3$ divides all the mentioned degrees. We will use all these facts in the following paragraphs.

Suppose first that $G$ contains a non-trivial field automorphism for $S$, let $\tau \in G$ be a field automorphism of maximal order and let $\nu=\tau^{o(\tau)/p}$ be a field automorphism of order $3$. Consider the group $\Gamma = \groupgen{\tau} \ltimes \tilde{S}$, so that $\Gamma = G\tilde{S}$. It follows from Corollary~\ref{corollary:semisimple_characters} that there exists a semisimple character $\psi \in \irr(\tilde{S})$ such that ${\psi^{\nu}}_S \neq \psi_S$ and either $\psi^{\sigma}=\psi$ or $\psi^{\sigma}=\psi^{\nu^k}$ for some positive integer $k$ coprime with $3$.

Since $\nu$ is of minimal order in the cyclic group $\groupgen{\tau}$, we have that $\psi$ has inertia subgroup $I_{\Gamma}(\psi)=\tilde{S}$. If $\tilde{S}<G$, then the thesis follows by Lemma~\ref{lemma:normal_coprime} (notice that, in this case, the conclusion also follows directly from \cite[Theorem~C]{Grittini:Degrees_of_irreducible_fixed_p-solvable}). Suppose now that $G \cap \tilde{S} = S$ and notice that $\psi_S$ is irreducible, since otherwise also $\psi^{\nu} \neq \psi$ would not restrict irreducibly to $S$, and we have seen that all the irreducible characters of $\tilde{S}$ but one restrict irreducibly to $S$. Hence, $\chi = \psi_S \in \irr(S)$, $\chi^{\nu} \neq \chi$ and either $\chi^{\sigma}=\chi$ or $\chi^{\sigma}=\chi^{\nu^k}$. In particular, $I_G(\chi)=S$ and it follows again from Lemma~\ref{lemma:normal_coprime} that $\chi^G$ is a $\sigma$-invariant irreducible character of $G$ of degree divisible by $3$.

We can now assume that $G$ does not contain a non-trivial field automorphism for $S$ and, thus, either $G=\tilde{S}$ or $G=S$. In the former case, $G$ has exactly $\frac{1}{6}(q^2 - 5q + 10)$ irreducible characters of degree $(q+1)(q^2+q+1)$. Since $q \equiv 1$ (mod $3$), we have that $q^2+q+1 \equiv 0$ (mod $3$). On the other hand, $q$ is congruent modulo $9$ to either $1$, $4$ or $7$ and we can see that, in either cases, $q^2 - 5q + 10 \not\equiv 0$ (mod $9$) and, thus, $3 \nmid \frac{1}{6}(q^2 - 5q + 10)$. It follows that $\sigma$ fixes at least one character of degree $(q+1)(q^2+q+1)$ and we are done in this case.

Finally, assume that $G=S$, then it has exactly $\frac{1}{3}(q - 4)$ irreducible characters of degree $q^2+q+1$ and $\frac{1}{6}q(q - 1)$ irreducible characters of degree $(q-1)(q^2+q+1)$. As in the previous paragraph, $3$ divides both degrees. Suppose that $9 \nmid q - 4$, then $3 \nmid \frac{1}{3}(q - 4)$ and $G$ has a $\sigma$-invariant character of degree equals to $q^2+q+1$. Hence, we can assume that $q \equiv 4$ (mod $9$) and it follows that $q(q - 1) \equiv 3$ (mod $9$). In particular, $3 \nmid \frac{1}{6}q(q - 1)$ and $G$ has a $\sigma$-invariant character of degree $(q-1)(q^2+q+1)$. This conclude our proof.
\end{proof}

We now cite a result from \cite{Grittini:Degrees_of_irreducible_fixed_p-solvable}, which we will use often, sometimes without mentioning it explicitly.

\begin{proposition}[{\cite[Proposition~2.3]{Grittini:Degrees_of_irreducible_fixed_p-solvable}}]
\label{proposition:invariant_in_subgroup}
Let $G$ be a finite group, let $p$ be a prime number and let $\sigma \in \gal(\mathbb{Q}_{\abs{G}} / \mathbb{Q})$ such that $o(\sigma)$ is a power of $p$. Suppose $H \leq G$ such that $p \nmid \abs{G:H}$ and suppose, for some $\psi \in \irr(H)$, that $\psi(1)$ divides the degree of every irreducible constituent of $\psi^G$. If $\psi$ is fixed by $\sigma$, then $\sigma$ also fixes some irreducible constituents of $\psi^G$.
\end{proposition}

Notice that the hypothesis of $\psi(1)$ dividing the degree of every irreducible constituent of $\psi^G$ is automatically verified if there exists $N \lhd G$, with $N \leq H$, and a character $\phi \in \irr(N)$ which extends to $\psi$.

We are now ready to prove Theorem~\ref{result:main} for almost-simple groups, in the stronger version we mentioned.

\begin{theorem}
\label{theorem:almost_simple}
Let $G$ be an almost-simple group with socle $S$, let $p$ be an odd prime number and let $\sigma \in \gal(\Q_{\abs{G}}/\Q)$ be of order $p$. If $p \mid \abs{G/S}$, then $G$ has a $\sigma$-invariant irreducible character $\chi$ of degree divisible by $p$ and such that $S$ is not contained in its kernel.
\end{theorem}

\begin{proof}
First, we notice that, if $S$ is either an alternating or a sporadic simple group, then the Classification of finite simple groups tells us that $\op{Aut}(S)/S$ is a $2$-group. Hence, we can assume that $S$ is a simple group of Lie type.

Let $S \leq \tilde{S} \leq \op{Aut}(S)$ be a regular embedding for $S$, and let $P/S \in \syl{p}{G/S}$. Suppose first that there exists a rational character $\phi \in \irr(S)$ of degree divisible by $p$ and let $T=I_G(\phi)$ be its inertia subgroup. We can assume that $T/S \cap P/S$ is a Sylow $p$-subgroup for $T/S$. Since $p$ is odd and $\phi$ is rational, it follows from \cite[Theorem~6.30]{Isaacs} that $\phi$ has a rational extension $\theta$ to $P \cap T$, and by Proposition~\ref{proposition:invariant_in_subgroup} there exists a $\sigma$-invariant character $\hat{\phi} \in \irr(T \smid \phi)$, lying over $\theta$, of degree divisible by $p$. We can take $\chi = \hat{\phi}^G$ and we are done in this case.

Hence, we can assume that no rational irreducible character of $S$ has degree divisible by $p$. It follows that either $p \nmid \abs{S}$ or we are in one of the cases of Theorem~\ref{result:simple_rational}. If $p=3$, $S = \op{PSL}_3(q)$ and $3 \mid q - 1$, or $S = \op{PSU}_3(q)$ and $3 \mid q + 1$, then $\abs{\tilde{S}/S} = 3$ and, since $\op{Aut}(S)/\tilde{S}$ is abelian, we have that $P$ is normal in $G$. Since $P/S$ is a $p$-group, by Proposition~\ref{proposition:PSL3} there exists a $\sigma$-invariant character $\psi \in \irr(P)$ of degree divisible by $p$ and such that $S \nleq \ker(\psi)$, then the theorem follows from Proposition~\ref{proposition:invariant_in_subgroup}.

In all the remaining cases, namely if $p \nmid \abs{S}$ or if we are in the remaining cases of Theorem~\ref{result:simple_rational}, we know from the Classification of finite simple groups that $p \nmid \abs{\tilde{S}/S}$, $P(G \cap \tilde{S})/S \lhd G/S$ and $P/S$ is generated by a field automorphism $\tau$ of order a power of $p$. By Corollary~\ref{corollary:semisimple_characters}, there exists $\psi \in \irr(\tilde{S})$ which does not have $S$ in its kernel, such that $[{\psi_S}^{\nu}, \psi_S] = 0$, where $\nu=\tau^{o(\tau)/p}$ has order $p$, and such that $\psi^{\rho}=\psi$, for $\rho \in \groupgen{\tau} \times \groupgen{\sigma}$ equal to either $1 \times \sigma$ or $\nu^k \times \sigma^{\shortminus 1}$ for some positive integer $k$ coprime with $p$. Let $N = G \cap \tilde{S}$, if $\phi$ is an irreducible constituent of $\psi_N$, then $\phi^{\nu} \neq \phi$ and, since $p \nmid \abs{\tilde{S}/N}$, $\phi^{\rho}=\phi$ and, thus, either $\phi^{\sigma}=\phi$ or $\phi^{\sigma}=\phi^{\nu^k}$. Since $\nu$ is contained in every proper subgroup of $\groupgen{\tau}$, we have that $I_{PN}(\phi)=N$ and it follows from Lemma~\ref{lemma:normal_coprime} that $\phi^{PN} \in \irr(PN)$ is $\sigma$-invariant and $p$ divides $\phi^{PN}(1)=\phi(1)\abs{PN:N}$. Finally, since $PN$ is normal in $G$, the theorem follows again from Proposition~\ref{proposition:invariant_in_subgroup}.
\end{proof}

\section{Proof of Theorem~\ref{result:main}}

The following lemma is essentially \cite[Lemma 5]{Bianchi-Chillag-Lewis-Pacifici:Character_degree_graphs}, only slightly more general.

\begin{lemma}
\label{lemma:extension}
Let $G$ be a finite group and let $M \lhd G$ such that $M=S_1 \times ... \times S_n$, for some non-abelian simple groups $S_i$. For $i=1,...,n$, let $\delta_i \in \irr(S_i)$ such that it extends to $\op{Aut}(S_i)$, and such that ${\delta_i}^g = \delta_j$ whenever ${S_i}^{g^{\shortminus 1}} = S_j$, for $g \in G$. Then, $\delta = \delta_1 \times ... \times \delta_n \in \irr(M)$ extends to $G$. Moreover, if each $\delta_i$ has a rational extension to $\op{Aut}(S_i)$, then $\delta$ has a rational extension to $G$.
\end{lemma}

\begin{proof}
Let $M = U_1 \times ... \times U_k$, where $U_1,...,U_k$ are minimal normal subgroups of $G$. In particular, each $U_j$ is the direct product of all the $S_i$s in some orbit of the action of $G$ on $S_1,...,S_n$. For any $j = 1,...,k$, we may write $\pi_j \subseteq \{1,...,n\}$ for the set of indexes such that $S_i \leq U_j$ if and only if $i \in \pi_j$. Now, for each $j = 1,...,k$, we consider the character $\phi_j = \xi_1 \times ... \times \xi_n \in \irr(M)$ such that $\xi_i=\delta_i$ if $i \in \pi_j$ and $\xi_i=1_{S_i}$ otherwise. Notice that, since $U_i \cap U_j = 1$ if $i \neq j$, $\prod_{j=1}^k \phi_j = \delta \in \irr(M)$.

Now, for some $j=1,...,k$ and for some $i = 1,...,n$ such that $S_i \leq U_j$, consider $\delta_i \in \irr(S_i)$ and notice that it has a (possibly rational) extension to $\no{G}{S_i}$, since the corresponding character $\hat{\delta_i} \in \irr(S_i\ce{G}{S_i}/\ce{G}{S_i})$ has a (possibly rational) extension to $\no{G}{S_i}/\ce{G}{S_i}$, which is isomorphic to a subgroup of $\op{Aut}(S_i)$ containing $S_i$. Let $\theta \in \irr(\no{G}{S_i})$ be such (possibly rational) extension and let $\chi_j = \theta^{\otimes G}$ as defined in \cite[Section~4]{Isaacs:Character_correspondence}. We have that $\chi_j$ is a character of $G$, and that it is rational if $\theta$ is rational. Moreover, by \cite[Lemma~4.1]{Isaacs:Character_correspondence}, if $x = x_1 \times ... \times x_n \in M$ then 
$$ \chi_j(x) = \prod_{t \in T} \theta(txt^{-1}) = \prod_{t \in T} {\delta_i}^t(x) = \prod_{m \in \pi_j} {\delta_m}(x) = \phi_j(x) $$
where $T$ is a set of representatives of the right cosets of $\no{G}{S_i}$ in $G$, since ${\delta_i}^t = \delta_m$ whenever ${S_i}^{t{\shortminus 1}}=S_m$, and $\{S_m\}_{m \in \pi_i}$ is an orbit under $G$ (notice that, with an abuse of notation, we are identifying $\delta_i$ with $\delta_i 1_{\ce{M}{S_i}} \in \irr(M)$). It follows that ${\chi_j}_M = \phi_j$.

Finally, if we consider $\chi = \prod_{i=1}^k \chi_i$, we have that $\chi$ is rational if each $\delta_i$ has a rational extension to $\op{Aut}(S_i)$, and that $\chi_M = \prod_{i=1}^k \phi_i = \delta \in \irr(M)$. Hence, $\chi$ is irreducible and it is a (possibly rational) extension of $\delta$ to $G$.
\end{proof}

The following proposition is a direct consequence of the previous Lemma~\ref{lemma:extension} and of \cite[Theorem 2]{Giudici-Liebeck-Praeger-Saxl-Tiep:Arithmetic_results}, and we will use it also to prove our Theorem~\ref{result:rational}. We recall that, if $p$ is a prime number, a group $H \leq \op{Sym}_n$ is said to be \textit{$p$-concealed} if $p \mid \abs{H}$ and all the orbits of $H$ on the power set of $\{1,...,n\}$ have size coprime with $p$.

\begin{proposition}
\label{proposition:rational}
Let $G$ be a finite group and let $M \lhd G$ be a non-abelian minimal normal subgroup, such that $M=S_1 \times ... \times S_n$, for some non-abelian simple groups $S_i \cong S$, for $i = 1,...,n$. Suppose $\rup{p'}{G/M}=G/M$ for some prime number $p$, then either 
\begin{itemize}
\item[a)] $G$ has a rational irreducible character of degree divisible by $p$,
\item[b)] $n=1$ and there exists $N \lhd G$ such that $N \cap M = 1$ and $G/N$ is an almost simple group with socle $MN/N$, or
\item[c)] there exists $M \leq N \lhd G$ such that either $G/N \cong \op{Alt}_k$ for some $k \mid n$, $G/N \cong \op{A\Gamma L}_1(8)$ and $p=3$, or $G/N \cong D_{10}$, the dihedral group of order $10$, and $p=2$.
\end{itemize}
\end{proposition}

\begin{proof}
If $n=1$, then $\ce{G}{M} \lhd G$, $\ce{G}{M} \cap M=1$ and $G/\ce{G}{M}$ is isomorphic to a subgroup of $\op{Aut}(S_1)$ and, thus, it is almost simple. Hence, from now on we assume that $n>1$.

We shall now consider the action of the group $G/M$ on the set $\Lambda = \{S_1, ... , S_n \}$. Let $\tilde{\Omega} \subseteq 2^{\Lambda}$ be a (eventually trivial) partition of $\Lambda$ such that $G/M$ preserves $\tilde{\Omega}$ and acts primitively on its elements. For $\Delta_i \in \tilde{\Omega}$, with $i = 1,...,k=\abs{\tilde{\Omega}}$, define
$$ U_i = \bigtimes_{S_j \in \Delta_i} S_j. $$
We have that $M = U_1 \times ... \times U_k$ and $G/M$ acts primitively on $\Omega = \{ U_1,...,U_k \}$. Moreover, if we consider $N = \bigcap_{i=1}^k \no{G}{U_i}$, we have that $G/N$ acts on $\Omega$ both primitively and faithfully, and $p \mid \abs{G/N}$ because $\rup{p'}{G/M}=G/M$.

Now, let $\delta \in \irr(S)$ be a rational character which extends rationally to $\op{Aut}(S)$. Notice that such character always exists: if $S$ is of Lie type we can use the Steinberg character, while if $S$ is either a sporadic or alternating group we can simply consider the restriction to $S$ of a rational character of $\op{Aut}(S)$ of odd degree. For $\Delta \subseteq \Omega$, we define the character
$$ \psi_{\Delta} = \delta_1 \times ... \times \delta_n \in \irr(M)=\irr(S_1 \times ... \times S_n) $$
such that $\delta_i = \delta$ if $S_i \leq U_j$ for some $U_j \in \Delta$, $\delta_i = 1_{S_i}$ otherwise. Let $T_{\Delta}=I_G(\psi_{\Delta})$ be the inertia subgroup of $\psi_{\Delta}$ in $G$ and notice that, by Lemma~\ref{lemma:extension}, $\psi_{\Delta}$ has a rational extension to $T_{\Delta}$.

Now, suppose that $G/N$ is $p$-concealed in its action on $2^{\Omega}$. Then, by \cite[Theorem 2]{Giudici-Liebeck-Praeger-Saxl-Tiep:Arithmetic_results}, either $G/N \cong \op{Alt}_k$, $G/N \cong \op{AGL}_3(2)$ or $\op{A\Gamma L}_1(8)$ and $p=3$, or $G/N \cong D_{10}$ and $p=2$. Since $\op{AGL}_3(2)$ has a rational character of degree $6$, we are done in this case. Assume now that $G/N$ is not $p$-concealed and, thus, there exists $\Delta \subset \Omega$ such that $p \mid \abs{\Delta^{G/N}} = \abs{G:G_{\Delta}}$, where $G_{\Delta}/N \leq G/N$ is the stabilizer of $\Delta$. Now, let $\psi_{\Delta}$ as in the previous paragraph and notice that $T_{\Delta}=G_{\Delta}$, since $g \in G$ fixes $\psi_{\Delta}$ if and only if ${U_j}^g \in \Delta$ whenever $U_j \in \Delta$. It follows that, if $\hat{\psi} \in \irr(T_{\Delta})$ is a rational extension of $\psi_{\Delta}$ to $T_{\Delta}$, then $\chi = \hat{\psi}^G \in \irr(G)$ is rational and $p \mid \chi(1) = \abs{G:T_{\Delta}} \hat{\psi}(1)$.
\end{proof}

\begin{proof}[Proof of Theorem~\ref{result:rational}]
First, notice that, if $\rup{2'}{G}$ has a rational character of even degree, then so has $G$ (this is a consequence of \cite[Theorem~6.30]{Isaacs}). Thus, we can assume that $\rup{2'}{G}=G$.

If $M \lhd G$ is solvable, the theorem follows by induction on $\abs{G}$. Let $M$ be a non-abelian minimal normal subgroup of $G$, it follows from Proposition~\ref{proposition:rational} that there exists $N \lhd G$ with $N \cap M = 1$ such that $G/N$ is almost simple, since alternating groups always have a rational irreducible character of even degree (see \cite[Theorem~C]{Tiep-TongViet:Odd-degree_rational}). By \cite[Theorem~B]{Tiep-TongViet:Odd-degree_rational}, we have that $G/N \cong \op{PSL}_2(3^f)$ or $\op{PGL}_2(3^f)$, for some odd $f \geq 3$, and we can see from \cite[Table~III]{Steinberg:The_representations_of} that $\op{PGL}_2(3^f)$ has a rational character of degree $3^f - 1$. Hence, $G/N \cong M \cong \op{PSL}_2(3^f)$ and $G=N \times M$. Then the theorem follows by induction on $\abs{G}$.
\end{proof}

The following lemma is a consequence of \cite[Problem~3D.4]{Isaacs:Finite_Group_Theory} and it was possibly already known. Notice that the lemma remains true, with an almost identical proof, also if we replace $p$ with a set of primes $\pi$.

\begin{lemma}
\label{lemma:frattini_quotient}
Let $G$ be a finite group such that $\rup{p}{G}=G$, for some prime number $p$, and suppose that $G$ acts by conjugation on a $p$-group $P$. If $G$ acts trivially on $P/\Phi(P)$, then $G$ centralizes $P$.
\end{lemma}

\begin{proof}
Let $C = \ce{G}{P}$ and notice that $C \lhd G$. Suppose $C \lneq G$, then there exists $xC \in G/C$ of prime order $r \neq p$. Hence, there exists $x \in G$ of order a power of $r$ and such that $x \not\in C$. However, $\groupgen{x}$ is a group of order coprime with $\abs{P}$ acting trivially on $P/\Phi(P)$ and it follows from \cite[Problem~3D.4]{Isaacs:Finite_Group_Theory} that $\groupgen{x} \leq C$, a contradiction.
\end{proof}

Finally, we also need a corollary of Theorem~\ref{result:simple_rational} and Proposition~\ref{proposition:PSL3}, to deal with the order of the \textit{Schur multiplier} $M(S)$ of a simple group of Lie type $S$ with no $\sigma$-invariant characters of degree divisible by $p$. See \cite[Definition~11.12]{Isaacs} for the definition of Schur multiplier.

\begin{corollary}
\label{corollary:Schur_multiplier}
Let $p > 2$ be a prime number, let $S$ be a simple group of Lie type of order divisible by $p$ and let $\sigma \in \gal(\Q_{\abs{S}}/\Q)$ of order $p$. If $S$ has no $\sigma$-invariant irreducible characters of degree divisible by $p$, then $p \nmid \abs{M(S)}$.
\end{corollary}

\begin{proof}
By Theorem~\ref{result:simple_rational} and Proposition~\ref{proposition:PSL3}, either $S=\PSL_n(q)$ or $\PSU_n(q)$ with $p>n$, or $S \in \{\PSL_4(q), \PSU_4(q), \tw2B_2(q^2), \PSO_{2n}^+(q), \PSO_{2n}^-(q)\}$, for some positive integer $n$ and some prime power $q$. In the latter case, we have from the Classification of finite simple groups that $M(S)$ is a $2$-group, and, thus, $p \nmid \abs{M(S)}$, with the only exception of $S=\PSU_4(3)$, with $\abs{M(S)}=36$, which, however, has a rational character of degree divisible by $3$. In the former case, we have that, with few exceptions, $\abs{M(S)}$ divides $n$ and, therefore, $p \nmid \abs{M(S)}$ also in this case. If these exceptions occur, the order of the Schur multiplier is equal to $2$, unless $S=\PSL_2(9)$, with $\abs{M(S)}=6$, $S=\PSL_3(4)$, with $\abs{M(S)}=48$, $S=\PSU_4(3)$, which we already treated, or $S=\PSU_6(2)$, with $\abs{M(S)}=12$. In all this cases, our corollary holds unless $p=3$ and, thus, $n=2$. Hence, the only possible exception is when $S = \PSL_2(9)$, and the Steinberg character is a rational irreducible character of $S$ of degree divisible by $3$, contradicting our hypothesis.
\end{proof}

\begin{proof}[Proof of Theorem~\ref{result:main}]
Assume first that $p=2$. We only need to prove that $G$ is solvable, since then the thesis follows from \cite[Theorem~A]{Grittini:Degrees_of_irreducible_fixed_p-solvable}. Let $M \lhd G$ be minimal such that $G/M$ is solvable and let $M/L$ be a chief factor of $G$; by eventually replacing $G$ with $G/L$ we can assume that $L=1$ and $M = S_1 \times ... \times S_n$ is a minimal normal subgroup, with $S_i \cong S$ for some non-abelian simple group $S$ and for each $i=1,...,n$. Since $G$ does not have any rational character of even degree, it follows from \cite[Theorem~B]{Tiep-TongViet:Odd-degree_rational} that $S \cong \op{PSL}_2(3^f)$ for some odd integer $f \geq 3$.

Let $q=3^f$ and notice that $q-1=2m$ for some odd integer $m$. Since $\sigma$ has order $2$, either it fixes all elements of $\mathbb{Q}_m$ or it acts on it as the complex conjugation. If we now consider the character table of $\op{PGL}_2(q)$ (see \cite[Table~III]{Steinberg:The_representations_of}) and we compare it with the one of $\op{PSL}_2(q)$ (see \cite[Theorem~38.1]{Dornhoff:Group_Representation_TheoryA}), we can see that $\op{PGL}_2(q)$ has $(q+1)/2$ real-valued characters of even degree having values in $\mathbb{Q}_m$, which are, thus, fixed by $\sigma$, and $(q-1)/2$ of them restrict irreducibly to $\op{PSL}_2(q)$. If $S \cong \op{PSL}_2(3^f)$ with $f$ odd, this observation implies that $S$ has an irreducible character $\delta \in \irr(S)$ of even degree with a $\sigma$-invariant extension to a subgroup $S < H < \op{Aut}(S)$ such that $H \cong \op{PGL}_2(3^f)$. Since $\abs{\op{Aut}(S)/H} = f$ is odd, $\delta$ has a $\sigma$-invariant extension to its inertia subgroup in $\op{Aut}(S)$.

Consider now the character $\phi=\delta \times 1_{S_2} \times ... \times 1_{S_n} \in \irr(M)$, with $\delta$ as in the previous paragraph, and let $T=I_G(\phi)$ be its inertia subgroup; we have that $T/\ce{G}{S_1}$ is the inertia subgroup in $\no{G}{S_1}/\ce{G}{S_1}$ of $\hat{\phi} \in \irr(S_1\ce{G}{S_1}/\ce{G}{S_1})$, corresponding to $\phi$, and the $\sigma$-invariant extension $\hat{\theta} \in \irr(T/\ce{G}{S_1})$ of $\hat{\phi}$ corresponds to a $\sigma$-invariant extension $\theta \in \irr(T)$ of $\phi$. Since $2 \mid \phi(1)=\theta(1)$, it follows that $\theta^G \in \irr(G)$ is a $\sigma$-invariant character of even degree, in contradiction with our hypothesis. Hence, $G$ is solvable if $p=2$ and we are done in this case.

From now on, we can assume that $p>2$. Let $M \lhd G$ be a minimal normal subgroup of $G$, then by induction we have that $\rup{p'}{G/M} = O/M \times K/M$, with $O/M = \rad{p}{G/M}$ and $K/M$ perfect such that $\rup{p'}{K/M}=K/M$, $\rad{p'}{K/M}$ is solvable, and $(K/M)/\rad{p'}{K/M}$ is the direct product of some non-abelian simple groups with no $\sigma$-invariant irreducible characters of degree divisible by $p$.

Suppose that $M$ is a $p$-group and notice that $M \leq \ze{O}$, since $\ze{O} \lhd G$ and $M \cap \ze{O} \neq 1$. Moreover, if $M \nleq \rup{p}{K}$, then $M \cap \rup{p}{K} = 1$ and $K = M \times \rup{p}{K}$ and we are done in this case. Hence, we can assume that $\rup{p}{K} = K$.

Suppose $M \leq \Phi(O)$, then we have that $K/M$ centralizes $O/\Phi(O)$ and, by Lemma~\ref{lemma:frattini_quotient}, it follows that $K/M$ centralizes $O$ and, thus, $M \leq \ze{K}$. Let $M \leq \tilde{H} \lhd G$ such that $\tilde{H}/M = \rad{p'}{K/M}$ and let $H$ be a $p'$-complement for $M$ in $\tilde{H}$, then $H = \rad{p'}{\tilde{H}} \lhd G$ and, by eventually replacing $G$ with $G/H$, we can assume that $\tilde{H}=M$ and $K/M$ is the direct product of some non-abelian simple groups of Lie type of order divisible by $p$. Let $S$ be any of these non-abelian simple groups, we have that $S$ has no $\sigma$-invariant irreducible characters of degree divisible by $p$ and, thus, it follows from Corollary~\ref{corollary:Schur_multiplier} that $p \nmid \abs{M(S)}$, where $M(S)$ is the Schur multiplier of $S$. Hence, we have by \cite[Corollary~11.20]{Isaacs} that $M$ has a normal complement in $\tilde{K}$ for every $M < \tilde{K} \lhd K$ such that $\tilde{K}/M$ is simple (namely, $\tilde{K} = M \times \tilde{K}'$) and, thus, $M$ has a normal complement $C$ in $K$. Moreover, $C=\rup{p}{K}$ and, therefore, $C$ is normal in $G$ and we are done in this case, since we have that $O \times K = O \times C$, with $C \cong K/M$.

Assume now that $M \cap \Phi(O) = 1$, then we have that $O = \tilde{O} \times M$ for some $\tilde{O} \lhd G$ and, by eventually replacing $G$ with $G/\tilde{O}$, we can assume that $O=M$. Let $C = \ce{K}{M} \lhd G$; by proceeding as in the previous paragraph we can prove that $M$ has a complement $\tilde{C}$ in $C$ such that $\tilde{C} \lhd G$. Hence, if $C=K$ we are done. If $C < K$, then by replacing $G$ with $G/\tilde{C}$ we can assume that $K/M$ is isomorphic to a subgroup of $\op{Aut}(M) = \op{GL}_k(p)$, for some $k$ such that $\abs{M}=p^k$. In particular, we have that $K$ is isomorphic to a subgroup of the affine group $\op{AGL}_k(p)$ and, thus, $M$ is complemented in $K$. Now, let $\lambda \in \irr(M)$ of order $p$ and let $T=I_K(\lambda)$ be its inertia subgroup; since $M$ has a complement in $T$, $\lambda$ has an extension $\hat{\lambda}$ to $T$ of order $p$ (see, for instance, \cite[Lemma~3.4]{Grittini:p-length_character_degrees}), which is $\sigma$-invariant, since it has values in $\Q_p$. It follows that, if $p \mid \abs{K:T}$, then $\hat{\lambda}^K \in \irr(K)$ is a $\sigma$-invariant character of degree divisible by $p$ and we get a contradiction. Hence, we must assume that the action of $K/M$ on $V=\irr(M)$ is $p$-exceptional, according to the definition in \cite{Giudici-Liebeck-Praeger-Saxl-Tiep:Arithmetic_results}. Moreover, by eventually replacing $K$ with $K/\tilde{M}$ for some $\tilde{M} < M$ such that $\tilde{M} \lhd K$, we can assume that $M$ is normal minimal in $K$ and, thus, $K/M$ acts irreducibly on $V$. Suppose first that the action is imprimitive and let $V = V_1 \oplus ... \oplus V_m$ be any imprimitive decomposition of $V$. By \cite[Theorem~3]{Giudici-Liebeck-Praeger-Saxl-Tiep:Arithmetic_results} and \cite[Theorem~2]{Giudici-Liebeck-Praeger-Saxl-Tiep:Arithmetic_results}, and since $K$ is perfect, we have that $K/M$ has a chief factor $K/N$ isomorphic to either $\op{Alt}_m$ or $\op{PSL}_3(2)$ and, in either cases, it follows that $K$ has a $\sigma$-invariant character of degree divisible by $p$. Suppose now that $K/M$ acts on $V$ primitively, then by \cite[Theorem~1]{Giudici-Liebeck-Praeger-Saxl-Tiep:Arithmetic_results} either $K/M$ is transitive on $V \smallsetminus \{1_M\}$ or it has a composition factor isomorphic to either $\op{PSL}_2(5)$, $\op{PSL}_2(11)$, $\op{M}_{11}$, or $\op{M}_{23}$. We can verify directly that each group appearing in the latter case has a $\sigma$-invariant irreducible character of degree divisible by $p$, for every choice of $\sigma$ of order $p$. Hence, we can assume that $K/M$ is transitive on $V \smallsetminus \{1_M\}$, and it follows from Hering's Theorem (see \cite[Appendix~1]{Liebeck:Affine_permutation_groups}) that, since $K$ is perfect and $p>2$, either
\begin{enumerate}[(i)]
\item $K/M$ has a normal subgroup isomorphic to $\op{SL}_a(q)$ with $p^k=q^a$, to $\op{Sp}_{2a}(q)$ with $p^k=q^{2a}$, or to $\op{SL}_2(5)$;
\item $K/M$ is isomorphic to a subgroup of $\op{GL}_2(p)$.
\item there exists $R \lhd K$ such that $K/R \cong \op{Alt}_5$;
\item $K/M \cong \op{SL}_2(13)$ and $p=3$;
\end{enumerate}
If (i) or (iii) hold, then $K$ has a chief factor $K/L$ isomorphic either to $\op{PSL}_a(q)$ or $\op{Sp}_{2a}(q)$, and the Steinberg character of such groups is a rational character of degree a power of $p$, or to $\op{PSL}_2(5)$ or $\op{Alt}_5$, which have a $\sigma$-invariant character of degree divisible by $p$ for every prime $p \mid \abs{K/L}$ and every $\sigma$ of odd order. The same is true also in case (ii), since we can see from \cite[II.8.27~Hauptsatz]{Huppert:Endliche_Gruppen} that the only non-solvable subgroups of $\op{GL}_2(p)$ are $\op{SL}_2(p)$ and possibly $\op{Alt}_5$. Finally, if (iv) holds, then $K/M$ has exactly two characters of degree $6$. Hence, our hypothesis is contradicted in each of these cases.

From now on, we can assume that $\rad{p}{G}=1$. Suppose that $M$ is an abelian $p'$-group and let $P < G$ be Sylow $p$-subgroup of $G$, then it follows from \cite[Theorem~B]{Grittini:Degrees_of_irreducible_fixed_p-solvable} that $1 < \no{G}{P} \cap M = \ce{M}{P} \leq \ce{M}{O} = \ze{O} \cap M$. Since $\ze{O} \lhd G$ and $M$ is minimal normal in $G$, we have that $M \leq \ze{O}$ and it follows that $O = \rad{p}{O} \times M$ and we are done.

Finally, suppose that $M = S_1 \times ... \times S_n$ is the direct product of some isomorphic non-abelian simple groups. Let $S \lhd M$ be any of such non-abelian simple groups and notice that, if $p \mid \abs{G : \no{G}{S}}$, then we can always find a character $\phi \in \irr(S)$ such that $\phi 1_{\ce{M}{S}} \in \irr(M)$ has a rational extension $\hat{\phi}$ to $\no{G}{S}=I_G(\phi 1_{\ce{M}{S}})$, and ${\hat{\phi}}^G$ is a rational irreducible character of degree divisible by $p$. Suppose $p \mid \abs{\no{G}{S} : \ce{G}{S}S}$, then by Theorem~\ref{theorem:almost_simple} there exist an irreducible character $\phi$ of $S$ and an irreducible character $\hat{\phi}$ of $I_G(\phi 1_{\ce{M}{S}}) \leq \no{G}{S}$, lying over $\phi 1_{\ce{M}{S}}$, such that $\hat{\phi}^{\no{G}{S}}$ is $\sigma$-invariant and of degree divisible by $p$, and it follows that $\hat{\phi}^G = (\hat{\phi}^{\no{G}{S}})^G$ is a $\sigma$-invariant irreducible character of degree divisible by $p$. Therefore, $\ce{G}{S}S/M$ contains a Sylow $p$-subgroup of $G/M$ for every choice of $S \lhd M$ simple, and it follows in particular that $O/M$ centralizes each $S_i$, for $i=1,...,n$, and, thus, $O = \ce{O}{M}M = \ce{O}{M} \times M$, and $\ce{O}{M} = \rad{p}{O} \lhd G$. Since we are assuming that $\rad{p}{G}=1$, we have that $O=M$. It follows that, by Proposition~\ref{proposition:invariant_in_subgroup}, $K = \rup{p'}{G}$ does not have any $\sigma$-invariant irreducible character of degree divisible by $p$ and, working by induction, we can assume that $K=G$. Since alternating groups of order divisible by $p$ always have a $\sigma$-invariant character of degree divisible by $p$, it follows from Proposition~\ref{proposition:rational}, and from $K$ being perfect, that $K = N \times M$ for some $N \lhd K$ and, by induction, $K/\rad{p'}{N} = N/\rad{p'}{N} \times M\rad{p'}{N}/\rad{p'}{N}$ is the direct product of some non-abelian simple groups with no $\sigma$-invariant irreducible characters of degree divisible by $p$.
\end{proof}

\begin{center}
\subsection*{Acknowledgements}
\end{center}

\noindent
The author is grateful to Gunter Malle for providing Theorem~\ref{result:simple_rational} and the contents of Section~\ref{section:simple_groups}, and in general for his precious advice and support in the research on the topic of the paper.

\bigskip

\printbibliography[heading=bibintoc]

\end{document}